\documentclass[letterpaper, 10 pt, conference]{ieeeconf}  
\usepackage{colortbl}
\definecolor{mygray}{gray}{.9}
\usepackage{cite}
\usepackage{color}
\usepackage{amsfonts,amssymb}
\usepackage{bm}
\usepackage{amsmath}
\usepackage{multirow}
\usepackage{algorithmic}
\usepackage{graphicx}
\usepackage{wrapfig}
\usepackage{subfigure}
\usepackage{float}
\usepackage{textcomp}
\usepackage{graphicx} 
\usepackage{epstopdf}
\usepackage{array}
\usepackage[thmmarks,amsmath]{ntheorem}
\newtheorem{mythm}{Theorem}
\newtheorem{mydef}{Definition}

\newtheorem{assumption}{Assumption}
\newtheorem{lemma}{Lemma}

\newtheorem{corollary}{Corollary}

\usepackage[title]{appendix}

\usepackage{array}
\usepackage{enumerate}
\usepackage{booktabs}
\usepackage{algorithm}
\usepackage{algorithmic}
\usepackage{setspace}
\usepackage{float}
\usepackage{cases}

\makeatletter\@mparswitchfalse\makeatother\normalmarginpar 
\usepackage[textwidth=1cm, textsize=scriptsize]{todonotes}

\usepackage{mathrsfs}

\IEEEoverridecommandlockouts                              
\renewcommand{\QED}{\QEDopen}
\DeclareRobustCommand{\qeds}{%
  \ifmmode \qed
  \else
    \leavevmode\unskip\penalty9999 \hbox{}\nobreak\hfill
    \quad\hbox{$\QED$}%
  \fi
}

\title{\LARGE \bf
Adaptive Mitigation of Insider Threats via Off-Policy Learning
}

\author{ 
Gehui Xu$^{1}$, Kaiwen Chen$^{2}$, Zhong-Ping Jiang$^{3}$, Thomas Parisini$^{4}$, and Andreas A. Malikopoulos$^{5}$ 
\thanks{This work is supported in part by NSF under Grants CNS-2227153, CNS-2401007, CMMI-2348381, IIS-2415478, in part by MathWorks, and in part by EPSRC Grant  EP/X033546/1.}
\thanks{$^{1}$G. Xu is with the 
School of Civil $\&$ Environmental Engineering, Cornell University, Ithaca, NY, USA. {\tt\small gx62@cornell.edu}}%
\thanks{$^{2}$K. Chen is with the Department of Electrical and Electronic Engineering, Imperial College London, UK.
        {\tt\small kaiwen.chen16@imperial.ac.uk}}
    \thanks{$^{3}$ Z. Jiang  is with the CAN Lab, New York University, Brooklyn, NY
11201 USA. {\tt\small zjiang@nyu.edu}}    %
        \thanks{$^{4}$T. Parisini is with  the Dept of Electrical and Electronic Engineering,
	Imperial College London, London SW7 2AZ, UK, and also with the Dept of Electronic Systems, Aalborg University, Denmark, and with the  Dept.
	of Engineering and Architecture, University of Trieste, Italy.
         {\tt\small t.parisini@imperial.ac.uk}}
        \thanks{$^{5}$A. A. Malikopoulos is with the Applied Mathematics, Systems Engineering, Mechanical Engineering, Electrical \& Computer Engineering, and School of Civil \& Environmental Engineering, Cornell University, Ithaca, NY, USA.   {\tt\small amaliko@cornell.edu}}
}

\begin{document}

\maketitle
\thispagestyle{empty}
\pagestyle{empty}

\begin{abstract}
An insider is a team member who covertly deviates from the team’s optimal collaborative strategy to pursue a private objective while still appearing cooperative. Such an insider may initially behave cooperatively but later switch to selfish or malicious actions, thereby degrading collective performance, threatening mission success, and compromising operational safety. In this paper, we study such insider threats within an insider-aware, game-theoretic formulation, where the insider interacts with a decision maker (DM) under a continuous-time switched system, with each time interval characterized by a distinct insider behavioral pattern or threat level. We develop a periodic off-policy mitigation scheme that enables the DM to learn optimal mitigation policies from online data when encountering different insider threats, without requiring a priori knowledge of insider intentions.
By designing appropriate conditions on the inter-learning interval time, we establish convergence guarantees for both the learning process and the closed-loop system, and characterize the corresponding mitigation performance achieved by the DM.

\end{abstract}

\section{INTRODUCTION}

An insider is a team member who covertly deviates from the team-optimal collaborative strategy in pursuit of a private objective, while maintaining an outward appearance of cooperation~\cite{xu2025game}. Depending on task requirements, environmental conditions, or personal incentives, such an insider may initially behave cooperatively but later switch to selfish or even malicious strategies at an unknown time, leaving other team members unaware of their true intentions. The hidden and time-varying nature of such behaviors makes insider threats particularly difficult to detect and mitigate, and the impact of insider-related incidents has become increasingly significant in recent years. According to a recent  report~\cite{PonemonDTEX2025}, the average annual cost associated with insider incidents increased by nearly 50$\%$ between 2019 and 2025. 

The resulting insider behaviors can severely degrade team performance and, in some cases, jeopardize team safety. For example, in a cooperative lane-change or lane-merge scenario~\cite{zhang2024stackelberg,rios2016survey}, a following vehicle may decelerate to create a merging gap but then covertly accelerate during the maneuver to sideswipe the merging vehicle, thereby shifting collision liability and facilitating insurance fraud.
As another example, in a human–robot collaboration task~\cite{sheng2025human,wang2022bounded}, the human may cooperate fully in the early stages but subsequently reduce their contribution or deviate from the shared goal, thus preserving energy. Such behaviors shift the workload to the robot and increase the risk of task failure.
Indeed, in microgrids, a malicious insider may leak sensitive topology information to external adversaries, enabling false-data injection attacks that manipulate voltage or current measurements and potentially trigger instability or widespread outages \cite{Gonen2020FDI}.  
Nevertheless, most existing research on safety and security in intelligent systems focuses on external adversaries, including secure coordination~\cite{farokhi2017private}, adversarial learning~\cite{chen2024approaching}, 
denial of service attack~\cite{Sayan_CSL2025},
and threat diagnosis mechanisms~\cite{zhang2025threat}. Insider threats arising from within cooperative teams remain largely underexplored.

Interactions between the insider and other team members are typically analyzed within game-theoretic frameworks~\cite{liu2020defense,liu2021flipit,hu2015dynamic,xu2024consistency}. However, most existing models assume that cooperative members have full knowledge of the insider’s behavior, an assumption that rarely holds in practice.  In our recent work~\cite{xu2025game}, we study unknown insider threats and develop an online indirect dual adaptive control approach to identify and mitigate such covert behaviors. 
Nevertheless, the insider behavior in \cite{xu2025game} is assumed to follow a fixed mode. In practice, the insider may covertly switch among different behavioral modes, such as cooperative, selfish, or even adversarial.
This motivates further investigation of insider threats with switching behaviors and mode uncertainty.

Regarding the problem of inferring unknown agent behaviors and strategies, a classical approach is to design an indirect adaptive optimal control law by first identifying the system parameters and then solving the associated optimality conditions, such as the Hamilton–Jacobi–Bellman (HJB) equation or the algebraic Riccati equation (ARE). However, adaptive systems designed this way
are known to respond slowly to parameter variations from the plant.  Reinforcement learning~\cite{sutton1998reinforcement} and approximate/adaptive dynamic programming (ADP)~\cite{werbos1974beyond}  have been widely applied to solve optimal control problems for uncertain systems in recent years~\cite{jiang2012computational}.  ADP
is a non-model-based method that can directly approximate the
optimal control policy via online learning~\cite{jiang2012computational}. The idea of ADP can be traced back to the seminal work by Paul Werbos~\cite{werbos1968elements} that brings reinforcement learning and dynamic programming together to solve Bellman’s equation.
Over the past two decades, attention has mainly been given to ADP-related control design for different systems~\cite{jiang2012computational,gao2016adaptive,bian2021reinforcement}. 

In this paper, we develop a periodic ADP learning scheme to mitigate insider threats with varying intentions.
Specifically, we formulate a game between an insider and a decision maker (DM) under strategic insider threats, and model their interaction using an unknown continuous-time switched system, where each time interval is characterized by a distinct behavioral pattern or threat level.
Mitigation policies under insider switching typically exhibit only partial steady-state regulation, which can lead to steady-state bias in learning and mitigation. To address this issue, we incorporate a Proportional--Integral augmentation to remove steady-state offsets, and further introduce incremental variables with a delay-based design to remove the influence of unknown mode-dependent disturbances induced by the insider's private objectives.
Then we develop a periodic off-policy mitigation scheme that enables the DM to learn optimal mitigation policies from online data, without requiring a priori knowledge of the system dynamics or the switching instants. By designing appropriate conditions on the inter-learning interval time, we establish convergence guarantees for both the learning process and the closed-loop system, and characterize the corresponding mitigation performance achieved by the DM.

The remainder of the paper is organized as follows. Section II introduces the team game model, Section III presents the off-policy learning–based mitigation approach, Section IV validates the framework via simulations, and Section VI concludes with future research directions.




\section{The Team Game Model}

\textbf{Notation.} Let \(\mathbb{N}^+\) denote the set of positive integers and \(\mathbb{R}\) denote the set of real numbers.
We use \(\mathbb{R}^{n}\) (or \(\mathbb{R}^{m\times n},\; m,n\in\mathbb{N}^+\)) to denote the set of \(n\)-dimensional real column vectors (or real \(m\)-by-\(n\) matrices). 
Let \(I_n\) denote the \(n \times n\) identity matrix, $\otimes$ denote the Kronecker product, and $\operatorname{vec}(A)\in\mathbb{R}^{mn}$ denote the column-wise vectorization of a matrix $A\in \mathbb{R}^{m\times n}$.  The logarithmic norm associated with $\|\cdot\|_2$ is defined as
$\mu_2(A)=\lim_{h\to0^+}(\|I+hA\|_2-1)/h$, which yields
$\mu_2(A)=\lambda_{\max}((A+A^\top)/2)$.
If $\mu_2(A)<0$, then $A$ is Hurwitz.

In this section, we model a nominal single-mode cooperative system as a two-player team game between a DM and a potential insider. Then, we characterize how different insider activations can alter the system’s evolution and, from the DM’s perspective, induce switched modes of behavior.

\subsection{Single-Mode Cooperative Model with Insider Threat}
Consider a two-player dynamic system  described by
\begin{equation}\label{dynamics}
\dot{x} = f(x) + g_1(x)u_1 + g_2(x)u_2,
\end{equation}
where $x(t)\in\mathbb{R}^n$ is the system state,
$f:\mathbb{R}^n\to\mathbb{R}^n$,
$g_1(x),g_2(x)\in\mathbb{R}^{n\times m}$ are smooth mappings,
and $u_1\in \mathbb{R}^m$ and $u_2\in \mathbb{R}^m$
are the control inputs of the DM and the insider, respectively.
Both players are assumed to have full observation of the system state.
The nominal collaboration model is formulated as a team game in which the DM and the insider cooperate and share the same information and a common objective to minimize:
\begin{align}\label{cooperative}
\mathcal{C}
=\int_0^\infty
& \left [ (x(t)-x_c^r)^\top Q_c (x(t)-x_c^r)
+ u_1(t)^\top R_1 u_1(t) \right .
\notag\\
&\left . + u_2(t)^\top R_2 u_2(t) \right ] \,dt,
\end{align}
where $Q_c \succ 0$ and $R_1, R_2 \succ 0$ 
denote real symmetric and positive definite  matrices,
and $x_c^r\in\mathbb{R}^n$ denotes the desired reference state associated with the collaborative objective. 
Common cost functions such as~\eqref{cooperative} may represent a variety of practical scenarios. For instance, in a lane-change scenario~\cite{zhang2024stackelberg,falsone2022lane}, the leading vehicle and the following vehicle aim to maintain a desired safety distance and a preferred speed. Or, in a human–robot interaction scenario~\cite{sheng2025human,wang2022bounded}, a robot and a human jointly move a heavy object to a desired location. 
Further details can be found in~\cite{xu2025game}.

We consider players making decisions under a dynamic closed-loop information structure, i.e., the DM and the insider choose instantaneous control strategies $u_1(x)$ and $u_2(x)$ based on the observed state $x$. A pair of feedback strategies $(u_1,u_2)$ is said to be admissible
if the resulting closed-loop system admits a (locally) asymptotically
stable equilibrium at $x=x_c^r$.
Then, the team outcome resulting from the joint decisions is characterized by the team-optimal solution~\cite{xu2025does,xu2025game}.
This solution corresponds to a strategy profile under which no unilateral or joint deviation by players can yield improved collective performance~\cite{radner1962team,zoppoli2020neural,Malikopoulos2021}.
\begin{mydef}\label{def:team}
A pair of feedback strategies $(u_1^*,u_2^*)$ is called a team-optimal solution to the cooperative problem \eqref{dynamics}--\eqref{cooperative} if, for all admissible feedback strategies $(u_1,u_2)$,
\[
\mathcal{C}(u_1^*,u_2^*;x_0)\le \mathcal{C}(u_1,u_2;x_0).
\]
\end{mydef}

In practice, the insider may appear to contribute to the team's task while secretly optimizing its own objective,
with the DM being unaware of this intention.
Instead of minimizing \eqref{cooperative}, the insider minimizes an alternative objective that blends a private goal with a penalty on deviations from
the nominal cooperative behavior:
\begin{align}\label{insider_obj}
\mathcal{C}^{\mathrm{adv}}_2
=\int_0^\infty
& \left [ (x(t)-x_a^r)^\top Q_a (x(t)-x_a^r)
+ u_2(t)^\top \tilde R_2 u_2(t) \right . 
\notag\\
& \left . +\rho\big(u_2(t)-u_2^{*}(t)\big)^\top \big(u_2(t)-u_2^{*}(t)\big) \right ] \,dt,
\end{align}
where $Q_a\succ 0$, $\tilde R_2\succ 0$ are insider-specific symmetric positive definite  matrices,
$x_a^r\in\mathbb{R}^n$ denotes the insider’s preferred reference state, and $\rho>0$ is a disciplinary risk parameter.
These parameters are unknown to the DM.
A larger value of $\rho$ corresponds to more cautious and concealed behavioral deviations,
whereas a smaller $\rho$ reflects more aggressive and overt adversarial actions.

Accordingly, knowing that the DM remains unaware of its true intention, the insider selects its best response to the DM's nominal team strategy $u_1^*$:
\[
u_2^\diamond
\in
\arg\min_{u_2}
\mathcal{C}^{\mathrm{adv}}_2(u_1^*,u_2;x_0).
\]
Without awareness of the insider threat, the DM continues to execute $u_1^*$ under the cooperation assumption,
potentially leading to performance degradation or even instability.

\subsection{Insider Activations and Switched Dynamics}

The above formulation focuses on the interaction with a fixed insider behavior.
However, in practice, the insider's intention may evolve.
Different intentions lead to distinct behavioral strategies and closed-loop responses. For example, a surrounding vehicle may covertly switch from cooperative to aggressive behavior in a lane-changing scenario, while a human operator may reduce effort to preserve energy in a collaborative task.

In more precise terms, consider a scenario with a finite number of behavioral phases. 
Let $\sigma(t)\in\mathcal M$ denote the switching signal representing the insider’s behavioral mode, which is piecewise constant.
Moreover, let $\{t_j\}_{j=0}^{J}$ denote the switching time-instants satisfying
$0 = t_0 < t_1 < \cdots < t_J$,
and define $t_{J+1}:=\infty$, so that no further switching occurs after $t_J$. The total number of switches $J<\infty$, and both $J$ and the switching instants $\{t_j\}$ are unknown to the DM.
The switching signal satisfies a minimum dwell-time condition
$\
t_{j+1}-t_j \ge \omega_j$,  $j=0,\dots,J,
$
for some constants $\omega_j>0$.
For each interval $[t_j,t_{j+1})$, define $\sigma_j\in\mathcal M$ as the constant mode, i.e.,
$
\sigma(t)=\sigma_j$,  $t\in[t_j,t_{j+1})$, $j=0,1,\ldots,J$.

Under the insider’s influence, the system evolves as a switched system. 
During each interval $[t_j,t_{j+1})$, the insider adopts a mode-dependent state-feedback policy
\begin{equation}\label{eq:insider_policy}
v(t) = v_j\big(x(t)\big), \qquad t\in[t_j,t_{j+1}),
\end{equation}
where $v_j(\cdot)$ is unknown to the DM and varies across modes due to changes in the insider’s objective. In particular, under a cooperative intention, $v_j(\cdot)$ coincides with the nominal team policy, whereas under selfish or adversarial intentions, it represents the insider’s best-response strategy with respect to its private objective.

Substituting \eqref{eq:insider_policy} into the system dynamics yields
\begin{equation}\label{eq:effective_switched_nonlinear}
\dot x
= f(x)
+ g_1(x)u_1
+ g_2(x)v_j(x),
\quad t\in[t_j,t_{j+1}).
\end{equation}

Equivalently, define the mode-dependent vector field
\[
F_j(x,u_1)
:= f(x)+g_1(x)u_1+g_2(x)v_j(x),
\]
so that
$\dot x(t) = F_j\big(x(t),u_1(t)\big)$.

Therefore, from the DM’s perspective, the physical process is an unknown switched system with unknown switching instants, with the switching induced by the evolution of the insider’s intention. The main question addressed in this paper is how the DM can design an effective mitigation strategy to counter such unknown insider threats.

\section{Off-policy-learning-based mitigation} 
In this section, we develop an off-policy ADP-based method to mitigate insider threats.
The subsequent design is carried out within a linear–quadratic (LQ) framework.

We begin by briefly reviewing the optimal feedback structure in a single mode.
This serves to clarify how an insider’s strategic response modifies the effective system
dynamics from the DM’s perspective, which in turn motivates the linear switched system formulation.
For LQ team games, we assume that   $(A, [B_1, B_2])$ is stabilizable. Together with $R_{1}\succ0$, $R_{2}\succ0$, and $Q_c\succ0$, it is well known that the problem admits a unique optimal feedback solution:
 $$   {u}_i^* = -{K}_i^* x - k_i^*, 
    \quad i \in \{1,2\},$$
where \(K_i^* \!=\! R_i^{-1} B_i^\top P^*\) 
and 
\(
k_i^* \!= \!-K_i^* x_c^r.
\)
The real symmetric matrix \(P^*\) is the positive definite solution to the ARE:
$
    A^\top \!P^* \!+\! P^*A \!+\! Q_c \!-\! P^* B R^{-1} B^\top\! P^*\! =\! 0
$
with \(B \!= \![B_1 \, B_2]\) and \(R \!=\! \operatorname{diag}(R_1, R_2)\).

By substituting \(u_2^*\) into the insider’s true objective function \(\mathcal{C}_2^{\text{adv}}\), 
and noting that \(Q_a \succ 0\), \(\tilde{R}_2 \succ 0\), and \(\rho>0\), 
the insider’s genuine optimal response to the DM’s action \(u_1^*\) 
retains a linear structure, provided the pair \((A - B_1 K_1^*,\, B_2)\) is stabilizable~\cite{anderson2007optimal}:
\begin{equation}\label{eq:true_feedback}
    {u}_2^\diamond = -{K}_2^\diamond x - k_2^\diamond,
\end{equation}
where $K_2^\diamond
= (\tilde{R}_2 + \rho I)^{-1}\big(B_2^\top \mathcal{P}^{\star}_2 + \rho K_2^*\big)$, ${k}_2^\diamond=-K_2^\diamond x_a^r$ and  $\mathcal{P}^{\star}_2$ is the solution of the following ARE 
\begin{align*}\label{eq:riccati_adversarial}
&  \quad (A-B_1K_1^*)^\top \mathcal{P}^{\star}_2 + \mathcal{P}^{\star}_2  (A-B_1K_1^*) + Q_a+\rho K_2^{*\top}  K_2^*\notag\\
   & - \!(\mathcal{P}^{\star}_2 B_2 \!+ \!\rho K_2^{*\top} )(\tilde{R}_2 \!+\! \rho I)^{-1}(B_2^\top \mathcal{P}^{\star}_2 \!+\! \rho K_2^*)      \!=\! 0.
    \end{align*}
Here, the reference state $x_a^r$ is designed such that the steady-state bias 
$(A-B_1K_1^*)x_a^r -B_1k_1^* $ in the closed-loop system $\dot{x}=(A-B_1K_1^*)x+B_2u_2-B_1k_1^*$ vanishes.

Substituting~\eqref{eq:true_feedback} into the system dynamics yields
\begin{equation*}\label{eq:dynamics_true}
    \dot{x} = A x + B_1 u_1 - B_2 {K}_2^\diamond x - B_2 k_2^\diamond.
\end{equation*}
From the DM’s perspective, the system evolves as
$\dot{x} = A_1 x + B_1 u_1 +d$,
where $A_1 = A -  B_2 {K}_2^\diamond$, $d = B_2 k_2^\diamond$, $A_1$ and $d$ satisfy $A_1x_m^r +d = 0$ with $x_m^r \in  \mathbb{R}^n$ denoting the  desired state under mitigation.  
\vspace{-0.2cm}

\subsection{Switched unknown dynamics induced by insider}

As the insider adjusts its strategy sequentially, the resulting interaction naturally evolves into a switched linear system characterized by unknown, time-varying modes:
\begin{equation}\label{eq:dm_mode_dynamics}
\dot x(t) = A_{\sigma_j} x(t) + B_1 u_1(t) + d_{\sigma_j},
\end{equation}
where $A_{\sigma_j}:=A-B_2K_{2,\sigma_j}^\diamond$ and $d_{\sigma_j}:=-B_2 k_{2,\sigma_j}^\diamond$ are unknown to the DM and depend on unknown $\sigma_j$.

Since $A_{\sigma_j}$ and $d_{\sigma_j}$ are unknown, the associated equilibrium point
$x_{\sigma_j}^{r}\in\mathbb{R}^{n}$ satisfying $A_{\sigma_j}x_{\sigma_j}^{r}+d_{\sigma_j}=0$, which corresponds to the desired mitigation objective, is also unknown.
In practice, the DM typically only knows partial information about $x_{\sigma_j}^{r}$. For instance, the safety distance is known to the DM, whereas the steady-state behavior required to achieve it (such as the corresponding velocity or posture) is unknown a priori.

To eliminate the dependence on the unknown components of $x_{\sigma_j}^{r}$, we
partition  it as 
$
x_{\sigma_j}^{r}
\!\!=\!
[x_{\sigma_j,d}^{r\top},
x_{\sigma_j,s}^{r\top}]^\top,
$
where $x_{\sigma_j,d}^{r}\!\in\!\mathbb{R}^{s}$ denotes the  available reference information, while $x_{\sigma_j,s}^{r}$ represents the unknown part.
The available information can be expressed as an output constraint
$
C x_{\sigma_j}^{r} \!= \!x_{\sigma_j,d}^{r},
$
where $C\in\mathbb{R}^{s\times n}$ selects the regulated components of the state whose references are known.
The equilibrium $x_{\sigma_j}^{r}$ is then uniquely determined by $A_{\sigma_j}x_{\sigma_j}^{r}\!+\!d_{\sigma_j}\!=\!0$ and $C x_{\sigma_j}^{r} \!= \!x_{\sigma_j,d}^{r}$.

On this basis,  we augment the state with an integral-type variable $z\in\mathbb{R}^{s}$ and define
$\xi:=[x;z]\in\mathbb{R}^{q}$ with $q=n+s$.
For each mode, this yields the augmented dynamics
\begin{equation}\label{eq:aug_mode_dynamics}
\dot \xi(t)
=
\mathcal A_{\sigma_j} \xi(t)
+
\mathcal B_{\sigma_j} u_1(t)
+
\bar d_{\sigma_j},
\end{equation}
where
\[
\mathcal A_{\sigma_j} \!=\!
\begin{bmatrix}
A_{\sigma_j} & E\\
C & 0_{s\times s}
\end{bmatrix},
\mathcal B_{\sigma_j}\!=\!\!
\begin{bmatrix}
B_{1,\sigma_j}\\
0_{s\times m_1}
\end{bmatrix},
\bar d_{\sigma_j}\!=\!\!
\begin{bmatrix}
d_{\sigma_j}\\
- x_{\sigma_j,d}^{r}
\end{bmatrix}.
\]
Here, $E\in\mathbb{R}^{n\times s}$ specifies how the integral
state $z$ is injected into the system dynamics.
Thus, the control policy can be expressed as
$
u_1(t) = -\mathcal K_j\,\xi(t),
$
without requiring explicit knowledge of  $x_{\sigma_j}^{r}$.
For notational simplicity, we write
$(\mathcal A_j,\mathcal B_j,\bar d_j)
:=
(\mathcal A_{\sigma_j},\mathcal B_{\sigma_j},\bar d_{\sigma_j})$ in the sequel.

For each mode $\sigma_j$, the mitigation-oriented cost  is
\begin{equation}\label{eq:mode_cost}
\mathcal C_{1,j}^{\mathrm{mit}}
=
\int_{0}^{\infty} \left [
(\xi-\xi_j^{r})^\top \mathcal Q (\xi-\xi_j^{r})
+
u_1^\top \tilde R_{1} u_1 \right ] \,\mathrm{d}t,
\vspace{-0.2cm}
\end{equation}
where $\mathcal Q\succ 0$ and $\tilde R_1\succ 0$ are designed symmetric positive definite  matrices
 and $\xi_j^{r}\in\mathbb{R}^q$ denotes the steady-state of the augmented system. Although each mode is active over a finite interval, the learned controller approximates the corresponding infinite-horizon LQR solution when the dwell time is sufficiently long. Note that  $\xi_j^{r}$ appears only in the cost objective and 
 is not required for controller synthesis \cite{young1972approach}. 
 
Since $\mathcal Q\succ 0$ and $\tilde R_1\succ 0$, provided that  \((\mathcal{A}_j, \mathcal{B}_j)\) is stabilizable, 
the optimal control for DM is 
$u_1=-\mathcal{K}_j^\star \xi=-[{K}_{x,j}^\star, \mathcal{K}_{z,j}^\star]\begin{bmatrix}
 x\\ z
\end{bmatrix}$
with $\mathcal{K}_j^\star=\tilde{R}_1^{-1}\mathcal{B}_j^{T}\mathcal{P}^\star_j$
where $\mathcal{P}^\star_j$ satisfies the following ARE equation:
\begin{equation}\label{mig_ARE}
    \mathcal{A}_j^\top \mathcal{P}_j^\star + \mathcal{P}_j^\star \mathcal{A}_j + \mathcal{Q} - \mathcal{P}_j^\star \mathcal{B}_j \tilde{R}_1^{-1} \mathcal{B}_j^\top \mathcal{P}^\star_j = 0;
\end{equation}
hence,
$(\mathcal{A}_j\!-\!\mathcal{B}_j \mathcal{K}_j^\star)\xi_j^{r}\!=\!-\bar{d}_j$.

If the insider’s intention is fully known, the problem reduces to numerically solving \eqref{mig_ARE} for each mode, for example, by using the classical Kleinman’s algorithm \cite{kleinman1968iterative}, which we briefly recall below for the single-mode case.

\begin{lemma}[\cite{kleinman1968iterative}]\label{kleinman}
Let $\mathcal{K}^0 \in \mathbb{R}^{m \times q}$ be any stabilizing feedback gain matrix, i.e.,
$\mathcal{A} - \mathcal{B}\mathcal{K}^0$ is Hurwitz. For $k = 0,1,\ldots$, let $\mathcal{P}^k \in \mathbb{R}^{q \times q}$ be the real symmetric positive definite solution of the Lyapunov equation  \begin{equation}
        \mathcal{A}^{k\top} \mathcal{P}^k + \mathcal{P}^k \mathcal{A}^k + Q + \mathcal{K}^{k\top} \tilde{R}_1 \mathcal{K}^{k} = 0,
        \label{eq:lyap_pk}
        \vspace{-0.2cm}
    \end{equation}
    where  $\mathcal{K}^{k}$ is computed recursively by   \begin{equation}
        \mathcal{K}^{k} =\tilde{R}_1^{-1} \mathcal{B}^\top \mathcal{P}^{k-1}.
        \label{eq:policy_update}
        \vspace{-0.2cm}
    \end{equation}
Then, the following properties hold:
\begin{enumerate}
    \item $\mathcal{A} - \mathcal{B} \mathcal{K}^{k}$ is Hurwitz;
    \item $\mathcal{P}^\star \preceq \mathcal{P}^{k+1} \preceq \mathcal{P}^k$;
    \item $\displaystyle \lim_{k \to \infty} \mathcal{K}^{k}= \mathcal{K}^\star$, 
          $\displaystyle \lim_{k \to \infty} \mathcal{P}^k = \mathcal{P}^\star$.
\end{enumerate}
\end{lemma}

The method described in Lemma~\ref{kleinman} is in fact a policy iteration method~\cite{kleinman1968iterative} for continuous-time linear systems. 
Given a stabilizing gain matrix $\mathcal{K}^k$, (\ref{eq:lyap_pk}) is known as the step of \emph{policy evaluation}, since it evaluates the cost matrix $\mathcal{P}^k$ associated with the control policy. 
Equation~(\ref{eq:policy_update}), known as \emph{policy improvement}, finds a new feedback gain $\mathcal{K}^{k+1}$ based on the evaluated cost matrix $\mathcal{P}^k$.

It is clear that perfect knowledge of $\mathcal{A}$ and $\mathcal{B}$ is required in Kleinman’s algorithm, 
which makes this method inapplicable in our setting. Moreover, the unknown switching instants and the presence of unknown constant disturbance $\bar d_j$ further complicate the controller synthesis.

To eliminate the influence of $\bar d_j$
and enable a data-driven design, we introduce the incremental variables 
$\Delta \xi(\tau) = \xi(t) - \xi(t-\tau)$,
$\Delta u_1(\tau) = u_1(t) - u_1(t-\tau)$,
for $t \geq \tau$, $\tau>0$.
When the system operates in mode $j$, the resulting incremental dynamics are given by
\begin{equation}
\Delta\dot{\xi}(\tau)
= \mathcal{A}_j \Delta \xi(\tau)
\!+\! \mathcal{B}_j \Delta u_1(\tau)
= (\mathcal{A}_j - \mathcal{B}_j \mathcal{K}_j)\Delta \xi(\tau).
\end{equation}
This transformation allows the subsequent controller synthesis to be formulated as a model-free ADP problem.
Under sufficiently long dwell time, each mode satisfies $\omega_j \gg \tau$, so that the delay $\tau$ does not span consecutive switching intervals except for negligible transients.
Under the learned optimal feedback $\mathcal K_j^\star$, the closed-loop dynamics in mode $j$ are
$\dot{\xi}(t)
= (\mathcal A_j-\mathcal B_j \mathcal K_j^\star)\,\xi(t) + \bar d_j$, resulting in  $\Delta\dot{\xi}(\tau)
= (\mathcal{A}_j - \mathcal{B}_j \mathcal{K}_j^\star)\Delta \xi(\tau)$.
Since $\mathcal A_j-\mathcal B_j \mathcal K_j^\star$ is Hurwitz, we have 
$\Delta \xi(\tau)\to 0$ during mode $j$. 
This implies that $\xi(t)$ converges to a constant vector. 
Moreover, the closed-loop system admits a unique equilibrium 
$\xi_j^r=-A_{\mathrm{cl},j}^{-1}\bar d_j$, we conclude that 
$\xi(t)\to \xi_j^r$.

We then develop a periodic off-policy method to obtain the optimal feedback gain $\mathcal{K}^\star_j$ when the system matrices $(\mathcal A_j,\mathcal B_j)$ and the switching instants $\{t_j\}$ are unknown.

\vspace{-0.3cm}

\subsection{Off-policy mitigation method}

The main idea of the off-policy method is to apply an initial control policy to the system over a finite number of time intervals and collect online measurements. All subsequent iterations are then conducted by repeatedly using the same batch of collected data.

Here, we adopt a periodic off-policy learning scheme for the case where the switching instants are unknown. Specifically, the DM periodically activates the off-policy learning phase, during which data are collected over a finite time window and used to iteratively update the feedback gain. Once the learning phase is completed, the DM implements the learned control policy to respond to potential insider threats during the fixed inter-learning interval, until the next learning phase is triggered.



We first present the off-policy scheme for single-mode and omit the subscript $\sigma(t)$ for simplicity.
Consider an arbitrary feedback control policy $u=u^0_1$ with
\begin{align}\label{sys_off}
\dot{\xi}=\mathcal{A}^k \xi+\mathcal{B}(\mathcal{K}^k\xi+u)+\bar{d} \, ,
\end{align}
where $\mathcal{A}^k =\mathcal{A}-\mathcal{B}\mathcal{K}^k$.
Then we have
$\Delta\dot{ \xi}
= \mathcal{A}^k \Delta \xi+ \mathcal{B} (\mathcal{K}^k\Delta \xi+\Delta u)$,
with $\Delta e_0=e_0(t)-e_0(\tau)$.
Taking the time derivative of $\Delta \xi(\tau)^\top \mathcal{P}^k \Delta \xi(\tau)$
 along the solutions of (\ref{sys_off}), 
 \begin{align}\label{inte1}
&\Delta \xi^\top(\tau+\delta \tau)\, \mathcal{P}^k\, \Delta \xi(\tau+\delta \tau)
- \Delta \xi^\top(\tau)\, \mathcal{P}^k\, \Delta \xi(\tau) \\
&=\! - \!\!\int_\tau^{\tau+\delta \tau}\!\!\!\!\!\!\! \Delta \xi^\top \mathcal{Q}^k \Delta \xi \, dt
\!+\!2 \!\!\int_\tau^{\tau+\delta \tau}\!\!\! \!\!(\mathcal{K}^{k}\Delta \xi\!+\!\Delta u)^\top \tilde{R }_1 \mathcal{K}^{k+1} \Delta \xi \, dt,\notag
\end{align}
where $\mathcal{Q}^k=\mathcal{Q}+\mathcal{K}^{k\top}\tilde{R }_1 \mathcal{K}^k$.

Here, $\Delta \xi(\tau)$ in \eqref{inte1} is generated from system (\ref{sys_off}), in which $\mathcal{K}^k$
is not involved. Therefore, the same amount of data collected on the interval $[\tau, \tau+\delta t]$ can
be used for calculating  $\mathcal{K}^k$, with $k=1,2,\dots$. This implementation is well-known as
off-policy learning, in that the actual policy used can be an arbitrary one, as
long as it keeps the solutions of the overall system bounded ~\cite{jiang2012computational}.

To facilitate parameter estimation, the unknown matrices $\mathcal{P}_{j+1}$ and $\hat{\mathcal{K}}_{j+1}$ are vectorized using the Kronecker product representation. The following identity is used:
\begin{equation}\label{Kronecker}
\operatorname{vec}(X Y Z) = (Z^{\top} \otimes X)\operatorname{vec}(Y).
\end{equation}
Thus, we have the following equalities:
\begin{align*}
&\Delta \xi^\top \mathcal{Q}^k \Delta \xi  = (\Delta \xi^\top \otimes \Delta \xi^\top)\operatorname{vec}(\mathcal{Q}^k) \, , \\
&(\Delta u + \mathcal{K}^{k} \Delta \xi)^{\top} \tilde{R }_1 \mathcal{K}^{k+1} \Delta \xi\\
&
=\!
\Big[
(\Delta \xi^{\top} \!\!\otimes\! \Delta \xi^{\top}\!)(I_n \otimes\mathcal{K}^{k\top} \!\tilde{R }_1)
\!+ \!(\Delta \xi^{\top} \!\!\otimes\! \Delta u^{\!\top}\!)(I_n \!\otimes \!\tilde{R }_1\!)
\Big] \mathcal{K}^{k+1} \, . \notag
\end{align*}
Further, for any positive integer $p$, we define matrices
$\delta_{\xi\xi} \in \mathbb{R}^{p \times q^2}$,
$I_{\xi \xi} \in \mathbb{R}^{p \times q^2}$,
and $I_{\xi u} \in \mathbb{R}^{p \times qm}$, such that
\begin{align}
&\delta_{\xi\xi}
\!=\!
\Big[
\Delta \xi \otimes \Delta \xi \big|_{\tau_1}^{\tau_1+\delta \tau},
\ldots,
\Delta \xi \otimes \Delta \xi \big|_{\tau_p}^{\tau_p+\delta \tau}
\Big]^{\top} \notag \\
&I_{\xi \xi}
=
\Big[
\int_{\tau_1}^{\tau_1+\delta t}\!\!\! \Delta \xi \!\otimes\! \Delta \xi \, d\tau,
\;
\ldots,
\;
\int_{\tau_p}^{\tau_p+\delta t} \Delta \xi \!\otimes\! \Delta \xi\, dt
\Big]^{\top}\notag\\
&I_{\xi u}
=
\Big[
\int_{\tau_1}^{\tau_1+\delta t} \Delta \xi  \otimes u^0_1 \, dt,
\;
\ldots,
\;
\int_{\tau_p}^{\tau_p+\delta t} \Delta \xi \otimes \Delta u \, dt
\Big]^{\top}.\notag
\end{align}
For any given stabilizing gain matrix $\mathcal{K}^{k}$, (\ref{inte1}) implies 
\begin{equation}\label{matrix_off}
\tilde{\Theta}^k
\begin{bmatrix}
\mathrm{vec}(\mathcal{P}^k) \\
\mathrm{vec}(\mathcal{K}^{k+1})
\end{bmatrix}
=
\tilde{\Xi}^k \, ,
\end{equation}
where $\tilde{\Theta}^k \in \mathbb{R}^{p \times (q^2 + qm)}$
and $\tilde{\Xi}^k \in \mathbb{R}^{p}$ are defined as
\begin{align*}
\tilde{\Theta}^k
&=
\Big[
\delta_{\xi\xi},
\;
-2 I_{\xi \xi}(I_n \otimes \mathcal{K}^{k\top} \tilde{R}_1)
-
2 I_{\xi u}(I_n \otimes \tilde{R}_1)
\Big] \, , \\
\tilde{\Xi}^k
&=
- I_{\xi \xi}\,\mathrm{vec}(\mathcal{Q}^k) \, .
\end{align*}
From the above, an initial stabilizing control
policy $u^0_1$ is applied and the online information is recorded in matrices $\delta_{\xi\xi}$, $I_{\xi \xi}$, and $I_{\xi u}$. Then, without requiring additional system information, the matrices $\delta_{\xi\xi}$, $I_{\xi \xi}$, and $I_{\xi u}$ can be repeatedly used to construct the data matrices $\tilde{\Theta}^k$ and $\tilde{\Xi}^k$ for iterations.

The following assumption is in place to ensure the uniqueness of the solution to \eqref{matrix_off}.
\begin{assumption}\label{off_unique}
There exists an integer $l > 0$ such that
\begin{equation}\label{rank}
\operatorname{rank}([I_{\xi \xi}\;I_{\xi u}])=\frac{q(q+1)}{2}+mq \, .
\end{equation}
\end{assumption}

\begin{lemma}[\cite{jiang2012computational}]
    Under Assumption~\ref{off_unique}, there exists a unique pair of matrices $(\mathcal{P}^k, \mathcal{K}^{k+1})$, with $\mathcal{P}^k = \mathcal{P}^{k\top}$ such that \eqref{matrix_off} is satisfied. 
\end{lemma}


{On this basis, 
we develop a periodic off-policy learning scheme.
Since switching may cause the collected data to include samples generated under different system dynamics, we restrict the learning duration such that each data segment spans at most two consecutive modes. By appropriately selecting the inter-learning interval, we ensure that at least one off-policy update is performed within each dwell-time interval. The controller obtained from mixed-mode data is used to initialize the controller for the currently active mode, and is subsequently refined through further learning toward an approximate optimal policy. Quantitative conditions relating the dwell time, learning duration, and inter-learning interval will be provided later.}





Define the feedback gain computed from the mixed-mode data based on modes $j$ and $j+1$
as $\hat{\mathcal K}_{j+1}$. The following lemma ensures that the unique solution of \eqref{matrix_off} corresponds  to the existence of an pair $(\hat{\mathcal{A}}_{j+1},\hat{\mathcal{B}}_{j+1})$.
\begin{lemma}
Under Assumption~\ref{off_unique},
if $(\hat{\mathcal P}_j,\hat{ \mathcal K}_{j+1})$ is the unique solution recovered from \eqref{matrix_off}, then there exists a pair $(\hat{\mathcal{A}}_{j+1},\hat{\mathcal{B}}_{j+1})$ such that $(\hat{\mathcal{P}}_j,\hat{ \mathcal K}_{j+1})$ satisfies
\begin{align}
&(\hat{\mathcal{A}}_{j+1}-\hat{\mathcal{B}}_{j+1} \hat{ \mathcal K}_{j})^\top \hat{\mathcal{P}}_j+\hat{\mathcal{P}}_j(\hat{\mathcal{A}}_{j+1}-\hat{\mathcal{B}}_{j+1} \hat{ \mathcal K}_{j})+\hat{\mathcal{Q}}_j=0,\notag\\
&\hat{ \mathcal K}_{j+1} = \tilde{R}_1^{-1}\hat{\mathcal{B}}_{j+1}^\top \hat{\mathcal P}_j, \label{eq:alg7}
\end{align}
where $ \hat{\mathcal{Q}}_j= \mathcal{Q}+ \hat{ \mathcal K}_{j}^\top\tilde{R}_1\hat{ \mathcal K}_{j}$.
\end{lemma}

\noindent{\textbf{Proof}.
Let
$(\hat{\mathcal{P}}_j,\hat{ \mathcal K}_{j+1})$ be a solution of \eqref{matrix_off}. Then under Assumption~\ref{off_unique}, $(\hat{\mathcal{P}}_j,\hat{ \mathcal K}_{j+1})$ are unique. 
Since $\hat{\mathcal{Q}}_j\succ0$ implies  $\hat{\mathcal{P}}_j$ nonsingular, we have
$\hat{\mathcal{B}}_{j+1}^\top \hat{\mathcal{P}}_j
=
(\tilde{R}_1  \hat{\mathcal{K}}_{j+1}\hat{\mathcal{P}}_j^{-1\top})\hat{\mathcal{P}}_j
=
\tilde{R}_1 \hat{\mathcal{K}}_{j+1}$.
Therefore, $\hat{\mathcal{B}}_{j+1}=\hat{\mathcal{P}}_j^{-1}\hat{\mathcal{K}}_{j+1}^\top \tilde{R}_1$.
Moreover, 
choose any matrix $S$ satisfying $S^\top \hat{\mathcal{P}}_j+\hat{\mathcal{P}}_j S=0$ (e.g., $S=0$), and define
$\hat{\mathcal{A}}_{j+1} := -\tfrac12 \hat{\mathcal{P}}_j^{-1}\hat{\mathcal{Q}}_j + S + \hat{\mathcal{B}}_{j+1} \hat{\mathcal{K}}_j
$. Then $\hat{\mathcal{A}}_{j+1}-\hat{\mathcal{B}}_{j+1} \hat{\mathcal{K}}_j = -\tfrac12 \hat{\mathcal{P}}_j^{-1}\hat{\mathcal{Q}}_j + S$ and
\begin{align*}
&(\hat{\mathcal{A}}_{j+1}-\hat{\mathcal{B}}_{j+1} \hat{\mathcal{K}}_j)^\top \hat{\mathcal{P}}_j+\hat{\mathcal{P}}_j(\hat{\mathcal{A}}_{j+1}-\hat{\mathcal{B}}_{j+1} \hat{\mathcal{K}}_j)
=\\
&\Big(-\tfrac12 \hat{\mathcal{Q}}_j\hat{\mathcal{P}}_j^{-1}+S^\top\Big)\hat{\mathcal{P}}_j
+
\hat{\mathcal{P}}_j\Big(-\tfrac12 \hat{\mathcal{P}}_j^{-1}\hat{\mathcal{Q}}_j+S\Big)\\
&
=
-\hat{\mathcal{Q}}_j + (S^\top \hat{\mathcal{P}}_j+\hat{\mathcal{P}}_j S)
=
-\hat{\mathcal{Q}}_j. 
\end{align*}
The proof is complete. \qeds 
}

On the other hand,
in single-mode learning process, a stabilizing feedback gain $\mathcal{K}^0$ is required as an a priori policy for the DM. Similarly, in the periodic framework, each learning phase requires an initial policy that stabilizes the currently active system.
However, under unknown switching, an initial stabilizing feedback gain $\mathcal{K}_{j+1}^0$ for mode $j+1$ may not be available a priori. 
The following assumption ensures that the optimal gain $\mathcal{K}_j^\star$ learned from clear mode-$j$ data can serve as a stabilizing initial gain $\hat{\mathcal{K}}^{0}_{j+1}$ for the learning via potentially mixed-mode data. The resulting gain $\hat{\mathcal{K}}^{\star}_{j+1}$ can, in turn, be used as an initial stabilizing one for the learning on mode $j+1$ when clear data becomes available.  The periodical mitigation scheme is shown in Fig.~\ref{fig11}.

\vspace{-0.25cm}

\begin{figure}[h]
	\centering	
			\includegraphics[width=\columnwidth]{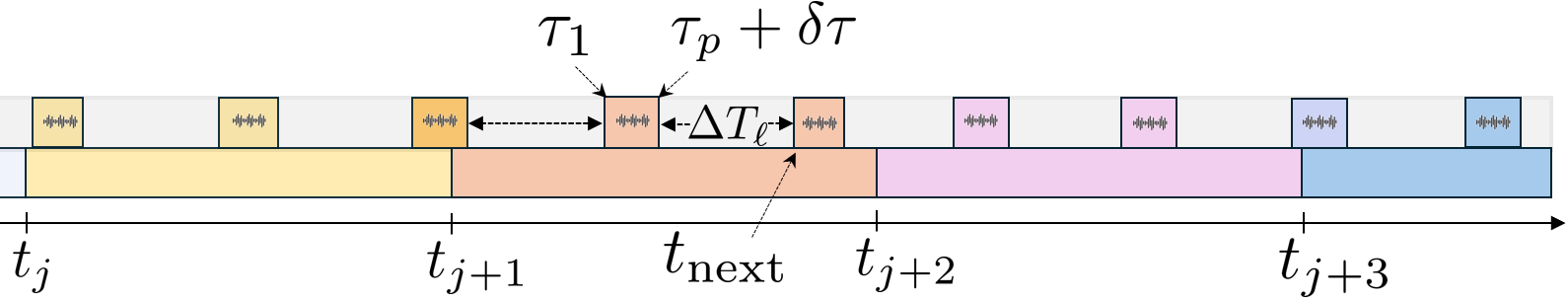}
\\\vspace{0.4cm}
\includegraphics[width=1\columnwidth]{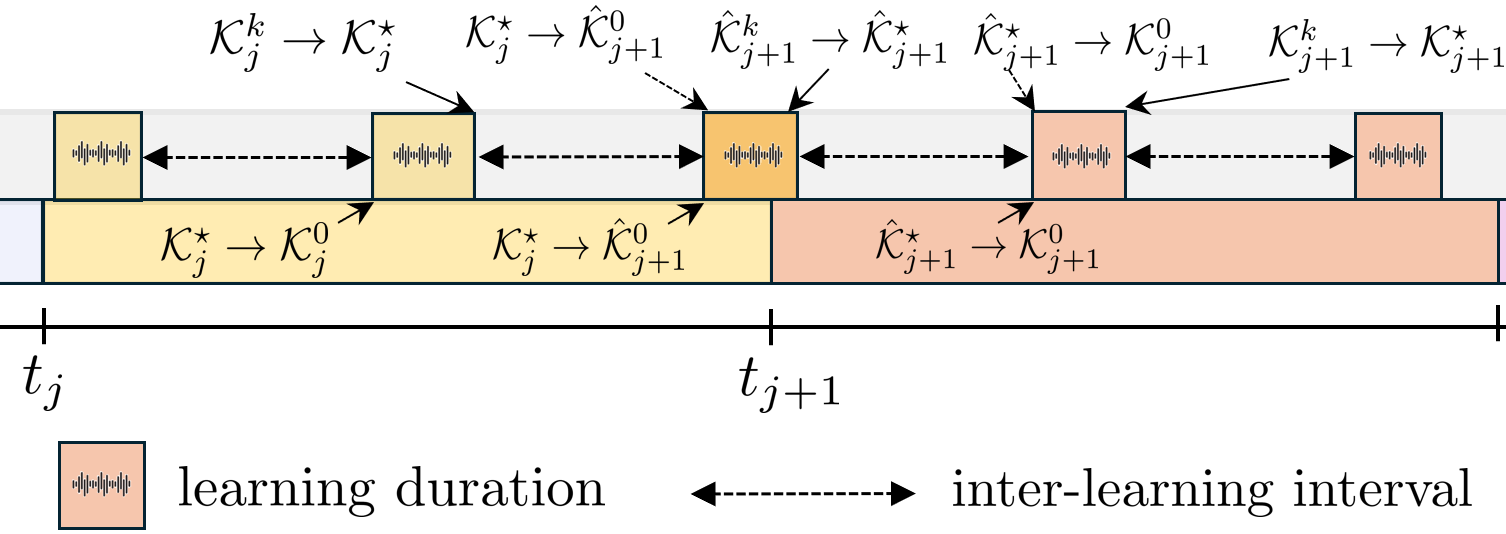}
\\\vspace{-0.1cm}
\caption{Periodical mitigation scheme} 
\vspace{-0.45cm}
	\label{fig11}
\end{figure}



\begin{assumption}\label{bound}
There exists $\Delta_r \ge 0$ such that for all $j,j+1\in\{0,\dots,J\}$,
$\|\xi_j^r-\xi_{j+1}^r\| \le \Delta_r $.
\end{assumption}

\begin{assumption}\label{ass:mu2_margin}
$~$
\begin{enumerate}[i)]
    \item There exists a known constant $\rho_j \ge 0$ such that
\begin{equation}
\left\|
(\mathcal{A}_{j+1}-\mathcal{A}_j) - (\mathcal{B}_{j+1}-\mathcal{B}_j)K_j^\star
\right\|_2
\le \rho_j .
\end{equation}
\item Fix the Euclidean norm $\|\cdot\|_2$ and its induced logarithmic norm
$\mu_2(M) := \lambda_{\max}\!\left(\frac{M+M^\top}{2}\right)$.
For each mode $j\in\mathcal M$, there exists a constant $\gamma_j>0$ such that the optimal closed-loop matrix satisfies
$\mu_2\!\left(\mathcal A_j-\mathcal B_j\mathcal K_j^\star\right)\le -\gamma_j$.
\item The input matrix in mode $j+1$ satisfies
$\|\mathcal B_{j+1}\|_2 \le \bar B $,
for some known constant $\bar B>0$.
\end{enumerate}

\end{assumption}




Moreover, the following lemma shows that the gain $\mathcal{K}_{j}^\star$ learned from clear-mode data provides a valid initialization for learning with mixed-mode data, and that the resulting gain $\hat{\mathcal{K}}_{j+1}^\star$ serves as a valid initialization for the subsequent learning phase in mode $j+1$.
\begin{lemma} \label{lem:mu2_hurwitz}
Suppose Assumption~\ref{ass:mu2_margin} holds. 
\begin{enumerate}
    \item If the  pair $(\hat{\mathcal A}_{j+1},\hat{\mathcal B}_{j+1})$
satisfies
$\big\|(\hat{\mathcal A}_{j+1}-\mathcal A_j)-(\hat{\mathcal B}_{j+1}-\mathcal B_j)\mathcal K_j^\star\big\|_2 < \gamma_j $,
then the mixed-mode closed-loop matrix
$
\hat{\mathcal A}_{j+1}-\hat{\mathcal B}_{j+1}\mathcal K_{j}^\star
$
is Hurwitz.
\item If
$\|\hat K_{j+1}-K_j^\star\|_2
<
\frac{\gamma_j-\rho_j}{\bar B}$,
$\rho_j<\gamma_j$,
the closed-loop matrix
$\mathcal A_{j+1}-\mathcal B_{j+1}\hat{\mathcal K}_{j+1}$
is Hurwitz.
\end{enumerate}
\end{lemma}
\begin{proof}
(i) Let
$\Delta_j := (\hat{\mathcal A}_{j+1}-\mathcal A_j)-(\hat{\mathcal B}_{j+1}-\mathcal B_j)\mathcal K_j^\star$, 
$\hat{\mathcal A}_{j+1}-\hat{\mathcal B}_{j+1}\mathcal K_j^\star
=
(\mathcal A_j-\mathcal B_j\mathcal K_j^\star)+\Delta_j$.
By the subadditivity of the logarithmic norm,
\[
\mu_2(\hat{\mathcal A}_{j+1}\!-\!\hat{\mathcal B}_{j+1}\mathcal K_j^\star)
\!\le\!
\mu_2(\mathcal A_j-\mathcal B_j\mathcal K_j^\star)
\!+\!\|\Delta_j\|_2
\!\le\!
-\gamma_j\!+\!\|\Delta_j\|_2.
\]
Thus, 
$\mu_2(\!\hat{\mathcal A}_{j+1}\!\!-\!\hat{\mathcal B}_{j+1}\mathcal K_j^\star)\!<\!0$, and
$\hat{\mathcal A}_{j+1}\!\!-\!\hat{\mathcal B}_{j+1}\mathcal K_j^\star$ is Hurwitz.

(ii)
Fix any $j\in\mathcal M$.
Define the previous closed-loop matrix
$A_\star := \mathcal A_j-\mathcal B_j\mathcal K_j^\star$,
and the candidate closed-loop matrix after switching to mode $j\!+\!1$ as
$A_{\hat K} := \mathcal A_{j+1}-\mathcal B_{j+1}\hat{\mathcal K}_{j+1}$.
Introduce the switching-induced perturbation (with respect to $\mathcal K_j^\star$)
$E_{j\to j+1}
:=
(\mathcal A_{j+1}-\mathcal A_j) - (\mathcal B_{j+1}-\mathcal B_j)\mathcal K_j^\star$,
and the controller mismatch
$\Delta K_{j+1} := \hat{\mathcal K}_{j+1}-\mathcal K_j^\star$.
Then we have
\begin{equation}\label{eq:Acl_decomp}
A_{\hat K}
=
A_\star
+
E_{j\to j+1}
-
\mathcal B_{j+1}\Delta K_{j+1}.
\end{equation}

By the definition of 
$\mu_2(X)$,
and its subadditivity property 
$\mu_2(X+Y)\le \mu_2(X)+\|Y\|_2$~\cite{hu2004weighted},
\[
\mu_2(A_{\hat K})
\le
\mu_2(A_\star)
+
\|E_{j\to j+1}\|_2
+
\|\mathcal B_{j+1}\Delta K_{j+1}\|_2.
\]
Under Assumption~\ref{ass:mu2_margin}, we have
$\mu_2(A_\star)\le -\gamma_j$ and $\|E_{j\to j+1}\|_2\le \rho_j$, and
$\|\mathcal B_{j+1}\|_2\le \bar B$. Hence,
\[
\|\mathcal B_{j+1}\Delta K_{j+1}\|_2
\le
\|\mathcal B_{j+1}\|_2\,\|\Delta K_{j+1}\|_2
\le
\bar B\,\|\hat{\mathcal K}_{j+1}-\mathcal K_j^\star\|_2.
\]
Therefore,
$\mu_2(A_{\hat K})
\le
-\gamma_j+\rho_j+\bar B\,\|\hat{\mathcal K}_{j+1}-\mathcal K_j^\star\|_2$.
If
$\|\hat{\mathcal K}_{j+1}-\mathcal K_j^\star\|_2
<
\frac{\gamma_j-\rho_j}{\bar B}$, 
$ \rho_j<\gamma_j$,
then $\mu_2(A_{\hat K})<0$.
\end{proof}

\medskip


We are now ready to present the periodic off-policy mitigation Algorithm~1. For notational simplicity, we omit the index $j$ and use the iteratively updated gains $\mathcal{K}^k$ and $\mathcal{K}^0$ to represent $\mathcal{K}_j^\star$ and $\hat{\mathcal{K}}_{j+1}^{0}$, respectively. The time interval $\tau_{p}+\delta \tau-\tau_1$ denotes the learning duration, which represents the minimum observation time required for the DM to infer the insider’s current behavior, and the time interval $\Delta T_{\ell}$ denotes the inter-learning interval. The off-policy learning procedure is activated periodically in a time-triggered manner, determined by $\tau_{p}+\delta \tau-\tau_1$ and $\Delta T_{\ell}$, where $t_{\operatorname{next}}$ denotes the starting time of the next learning phase.

\begin{algorithm}[h]
	\caption{Off-policy mitigation algorithm}
	\label{off-alg}
	\begin{algorithmic}[1]
  \renewcommand{\algorithmicrequire}{
  \textbf{Initialize:}}
		\REQUIRE Set $\Delta T_{\ell}$ and  $\mathcal{K}^0$ such that $\mathcal{A}_1-\mathcal{B}_1\mathcal{K}^0$ is Hurwitz.  Let  $q=0$, $j=0$, $t_{\operatorname{next}}=\tau$.
        \IF{$t\geq t_{\operatorname{next}}$}
 \STATE 
  {\textbf{Online data collection:}}  
		 Apply $u=-\mathcal{K}^0\xi+e_0$ to the system from $\tau_{1}=t$. Compute $\delta_{\xi\xi},I_{\xi \xi}$, and $I_{\xi u}$ until the rank condition \eqref{rank}  is satisfied. Let $k=0$.
        \WHILE{$\|\mathcal{P}^k-\mathcal{P}^{k-1}\|\geq \epsilon$}
		\STATE \textbf{Policy evaluation and improvement:}  Solve $\mathcal{P}^k$ and $\mathcal{K}^{k+1}$
     from \eqref{matrix_off}
     \STATE $k=k+1$;  
     \ENDWHILE

     \textbf{Post-learning implementation}: 
    Apply $u=-\mathcal{K}^{k}\xi$ as the control during $\Delta T_{\ell}$

    $t_{\operatorname{next}}\gets t_{\operatorname{next}}+\tau_{p}+\delta \tau-\tau_1+\Delta T_{\ell}$    

     $\mathcal{K}^0\gets\mathcal{K}^{k}$
    \ENDIF
	\end{algorithmic}
\end{algorithm}
\vspace{-0.35cm}

\subsection{Convergence analysis}
{Next, we investigate the convergence of Algorithm~1.
For each mode $\sigma_j$ with $j\in\{0,1,\ldots,J\}$, 
denote $\omega_{\min}:=\min_{j}\omega_j$ as the minimum dwell time
and  $\zeta_{\max}:=\max_j \tau_{p,j}+\delta\tau-\tau_{1,j}$ as the maximum learning duration required to accumulate sufficient online data to satisfy the rank condition \eqref{rank}. Define $\gamma_{\min}:=\min_{j}\gamma_j$. 
The following result can be established. }

\begin{mythm}\label{thm:dwell_stability}
Suppose that Assumptions~\ref{off_unique}--\ref{ass:mu2_margin} hold.
{ If the inter-learning interval $\Delta T_{\ell}$ satisfies $\Delta T_{\ell}<\omega_{\min}-\frac{\ln(2\nu)}{\gamma_{\min}}-2\zeta_{\max}$ with $\nu\geq1$}, then the sequence ${\mathcal{K}^k}$ generated during the learning phase for mode $\sigma_j$ converges to the optimal gain $\mathcal{K}^{\star}_j$ associated with $\mathcal{P}^{\star}_j$, for $j=0,\dots,J$.
Moreover, there exist constants $c_1>0$, $c_2\ge 0$, and $\alpha>0$ such that the closed-loop trajectory satisfies
\[
\|\xi(t)-\xi_{\sigma(t)}^r\|
\le
c_1 e^{-\alpha t}\|\xi(0)-\xi_{\sigma(0)}^r\| + c_2,
\qquad \forall t\ge 0 .
\]

\end{mythm}

\begin{proof}
The convergence of a single off-policy learning phase follows the proof in \cite{jiang2012computational}. 
Briefly, the matrices $(\mathcal{P}^k, \mathcal{K}^{k+1})$ obtained from Kleinman’s algorithm satisfy the condition in \eqref{matrix_off}. 
Conversely, under Assumption~\ref{off_unique}, any pair $(\mathcal{P}^k, \mathcal{K}^{k+1})$ that satisfies \eqref{matrix_off} also corresponds to the update steps \eqref{eq:lyap_pk} and \eqref{eq:policy_update} of Kleinman’s algorithm. 
Therefore, when Kleinman’s algorithm is initialized with a stabilizing gain $\mathcal{K}^0$, each subsequent iteration remains stabilizing, which guarantees the closed-loop stability throughout the learning process.
Furthermore, Lemma~\ref{lem:mu2_hurwitz} establishes the stability linkage between successive learning phases, thereby ensuring the convergence of the periodic learning scheme.

{Since $\Delta T_{\ell}+2\zeta_{\max}<\omega_{\min}-\frac{\ln(2\nu)}{\gamma_{\min}}$, at least one off-policy update within each dwell-time interval. This guarantees that  $\mathcal K_j^\star$ can be learned from the clear mode-$j$
 data.
Let $s_j\in[t_j,t_{j+1})$ denote the first time instant at which the controller
 is switched to the mode-wise optimal gain $\mathcal K_j^\star$ and $
 \delta_j := t_{j+1}-s_j  $ denote the effective optimal dwell-time. We have $\Delta T_{\ell}+2\zeta_{\max}<t_{j+1}-t_j-\frac{\ln(2\nu)}{\gamma_{\min}}$ and $s_j<t_j+\Delta T_{\ell}+2\zeta_{\max}<t_{j+1}-\frac{\ln(2\nu)}{\gamma_{\min}}$. Thus,
 $\delta_j> \frac{\ln(2\nu)}{\gamma_{\min}}$,}

We next prove the system stability.
For $t\in[s_j,t_{j+1})$, define $\tilde\xi_j(t)=\xi(t)-\xi_j^r$ with the error dynamics 
\[
\vspace{-0.2cm}
\dot{\tilde\xi}_j
=(\mathcal A_j-\mathcal B_j\mathcal K_j^\star)\tilde\xi_j .
\]
By Assumption~\ref{ass:mu2_margin}(ii), 
$\frac{d}{dt}\|\tilde\xi_j(t)\|_2
\!\le\!
\mu_2(\mathcal A_j\!-\!\mathcal B_j\mathcal K_j^\star)\|\tilde\xi_j(t)\|_2
\!\le\!
\!-\!\gamma_j\|\tilde\xi_j(t)\|_2$.
Hence,
$
\|\tilde\xi_j(t_{j+1}^-)\|_2
\le
e^{-\gamma_j\delta_j}\|\tilde\xi_j(s_j^+)\|_2 .
$

Define the Lyapunov function
$
V_j(t):=\tilde\xi_j(t)^\top \mathcal P_j^\star \tilde\xi_j(t),
$ $t\in[t_j,t_{j+1}),
$
with $\mathcal P_j^\star \succ 0$.
At $t_{j+1}$, $\xi$ is continuous but the reference changes, so
$
\tilde\xi_{j+1}(t_{j+1}^+)
=
\tilde\xi_j(t_{j+1}^-)+(\xi_j^r-\xi_{j+1}^r).
$
Using $\|a+b\|_2^2\le 2\|a\|_2^2+2\|b\|_2^2$,
$
\|\tilde\xi_{j+1}(t_{j+1}^+)\|_2^2
\le
2\|\tilde\xi_j(t_{j+1}^-)\|_2^2
+2\|\xi_j^r-\xi_{j+1}^r\|_2^2.
$
Therefore,
\[
\begin{aligned}
V_{j+1}(t_{j+1}^+)
&\le
\lambda_{\max}(\mathcal P_{j+1}^\star)\|\tilde\xi_{j+1}(t_{j+1}^+)\|_2^2\\
&
\le
2\lambda_{\max}(\mathcal P_{j+1}^\star)\|\tilde\xi_j(t_{j+1}^-)\|_2^2
+ c_{r,j},
\end{aligned}
\]
where
$
c_r=2\,\lambda_{\max}(\mathcal P_{j+1})
\|\xi_j^r-\xi_{j+1}^r\|_2^2 .
$

Take $\nu := \max_{\sigma_i,\sigma_j\in\mathcal M} $$\lambda_{\max}\!\left(\mathcal P_i^{-1}\mathcal P_j\right) \ge 1$, we have
$
\lambda_{\max}(\mathcal P_{j+1}^\star)
\le
\nu\lambda_{\max}(\mathcal P_j^\star),
$
and 
$V_{j+1}(t_{j+1}^+)
\le
2\nu e^{-2\gamma_j\delta_j}\,V_j(s_j^+) + c_{r,j}$. 
Since $V_j(s_j^+)\le V_j(t_j^+)$,
we obtain
\begin{equation}\label{decrease}
V_{j+1}(t_{j+1}^+)
\le
\rho_j\,V_j(t_j^+) + c_{r,j},
\qquad
\rho_j=2\nu e^{-2\gamma_j\delta_j}.
\end{equation}
By the effective dwell-time condition,
$\rho_j\le \rho:=2\nu e^{-2\gamma_{\min}\delta_{\min}}<1$,
where $\delta_{\min}:=\min_j\delta_j$. 
Equation \eqref{decrease}
gives $V_j(t_j^+)
\le
\rho^j V_0(t_0^+) + \frac{\bar c_r}{1-\rho}$ with  $\bar c_r:=\max_j c_{r,j}$.

Using $V_j(t)\ge \lambda_{\min}(\mathcal P_j^\star)\|\tilde\xi_j(t)\|_2^2$
and $V_0(t_0^+)\le \lambda_{\max}(\mathcal P_0^\star)\|\tilde\xi_0(t_0^+)\|_2^2$,
we obtain
\[
\|\xi(t)-\xi_{\sigma(t)}^r\|_2
\le
c_1 e^{-\alpha t}\|\xi(0)-\xi_{\sigma(0)}^r\|_2 + c_2,
\]
where $\alpha = \frac{-\ln\rho}{2\delta_{\min}}$, $c_1 := \sqrt{\frac{\overline\lambda}{\underline\lambda}}$ and $c_2 := \sqrt{\frac{2\overline\lambda}{\underline\lambda}}\;
\frac{\Delta_r}{\sqrt{1-\rho}}$ with $\underline\lambda := \min_{j}\lambda_{\min}(\mathcal P_j^\star),
$ $
\overline\lambda := \max_{j}\lambda_{\max}(\mathcal P_j^\star)$
\end{proof}

Accordingly, when the switching terminates after $t_J$, yielding $\Delta_r = 0$, the following conclusion holds.

\begin{corollary}\label{cor:final_convergence}
Under the conditions of Theorem~\ref{thm:dwell_stability}, if the switching terminates after $t_J$,
i.e., $\sigma(t)\equiv \sigma_J$ for all $t\ge t_J$,
then there exist constants $c_f>0$ and $\alpha_f>0$ such that
\[
\|\xi(t)-\xi_J^r\|
\le
c_f e^{-\alpha_f (t-t_J)}\|\xi(t_J)-\xi_J^r\|,
\qquad \forall t\ge t_J,
\]
and hence $\lim_{t\to\infty}\xi(t)=\xi_J^r$.
\end{corollary}

\section{Simulation Studies}

In this section, we validate the performance of the proposed mitigation scheme using the lane-change scenario~\cite{xu2025game}. 

Consider a highway scenario with two vehicles where the leading vehicle aims to maintain a safe distance of 
\(\pi = 73\,\mathrm{m}\) \cite{highwaycode126} from the following vehicle during the lane-change maneuver.  An insider vehicle may either behave cooperatively to accommodate the lane change or act selfishly for personal gain. In the cooperative mode, the desired spacing is set to \(\pi = 73\,\mathrm{m}\), whereas in the adversarial mode aimed at inducing a collision,
 \(\pi = 0\,\mathrm{m}\). 
Accordingly, we consider three switching modes: one cooperative mode and two various insider-threat modes.

In this case, DM only knows the desired distance is $x_{j,d}^r=\pi\in\mathbb{R}$ but the desired velocity $x_{j,s}^r\in\mathbb{R}^2$  is unknown. 
The additional state $z$ is designed as
$\dot z=Cx-\pi$ with  $C=[1,0,0]$. Under the three insider behavior modes, the system dynamics evolves as 
$\dot x \!=\! A_{1,\sigma_j} x \!+\! B_1 u_1+d_{\sigma_j},
B_1 =[0;1;0]$,
with $A_{\sigma_1}=\begin{bmatrix}
 0 & 1 & -1\\
 0 & 0 & 0\\
 0.5345        & 0.2893 & -1.84770\\
\end{bmatrix}$, $A_{\sigma_2}=\begin{bmatrix}
 0 & 1 & -1\\
 0 & 0 & 0\\
 0.240          & 0.35 & -1.60\\
\end{bmatrix}$, $A_{\sigma_3}=\begin{bmatrix}
 0 & 1 & -1\\
 0 & 0 & 0\\
 0.225           & 0.30 & -1.80\\
\end{bmatrix}$ and $d_{1,\sigma_1}=[0;0;18]$,  $d_{1,\sigma_1}=[0;0;25]$, $d_{1,\sigma_1}=[0;0;30]$.

The delay time $\tau=0.1$s. The length of each learning interval is 0.02s, and the exploration noise is selected to be a sinusoidal signal
$
\epsilon(t) = \sum_{k=1}^{100} \sin(\omega_k t),
$
where $\omega_k$, for $k = 1,\ldots,100$, are randomly generated frequencies uniformly distributed over $[-50,\,50]$. Before the learning phase, the system is operated for $2\tau$ seconds as a warm start to fill the delay window, after which data collection starts.  The dwell time $\omega_j=24$s.
\begin{figure}[H]
	\centering	
    \vspace{-0.2cm}
		\includegraphics[width=1\columnwidth]{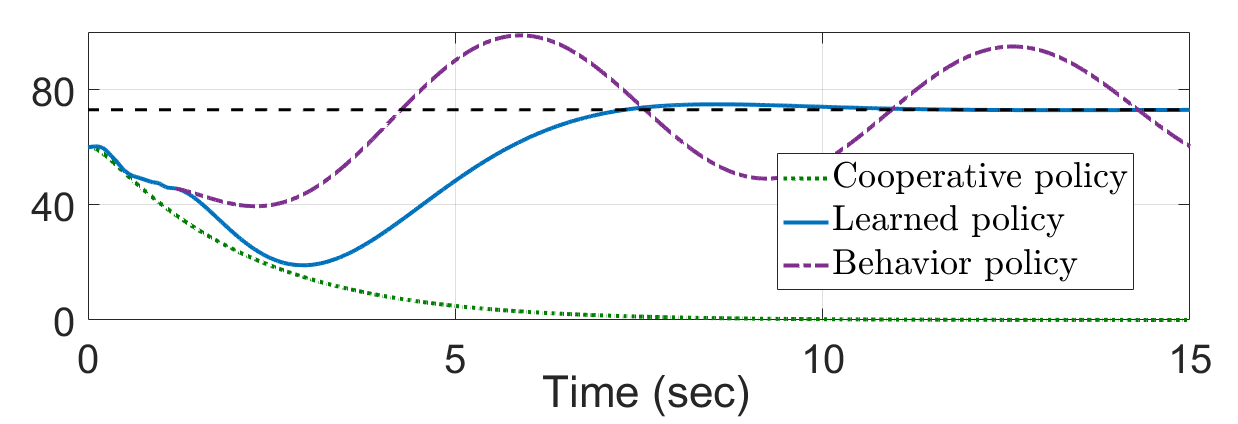}
\\
\vspace{-0.2cm}
	\caption{State trajectories switching
between three operating modes.}
\vspace{-0.3cm}
	\label{fig38}
\end{figure}

We first consider a single-mode adversarial insider aiming to induce a collision.
Fig.~\ref{fig38} compares the inter-vehicle distance resulting from the learned mitigation policy (blue) with those under the initial behavior policy (purple) and the nominal cooperative policy (green). The learned policy maintains a safe spacing, whereas the other two policies violate the safety constraint. In extreme cases, the cooperative policy drives the spacing to zero, leading to a collision.

Next, we consider the case where the insider behavior switches among three different modes.
By setting 
$\Delta T_{\ell}=10$s and $\Delta T_{\ell}=8$s, respectively, at least two learning updates are performed within each switching interval. This leads to two types of updates: those based on all clear-mode data and those based on partially mixed-mode data.
Fig.~\ref{fig17} shows the resulting system performance. When switching occurs, learning with all clear-mode data yields smaller trajectory oscillations. Although the update from mixed-mode data is biased relative to the active mode, the subsequent update using clear-mode data compensates for this bias, allowing the DM to maintain a safe headway.

\begin{figure}[H]
	\centering	
			\includegraphics[width=1\columnwidth]{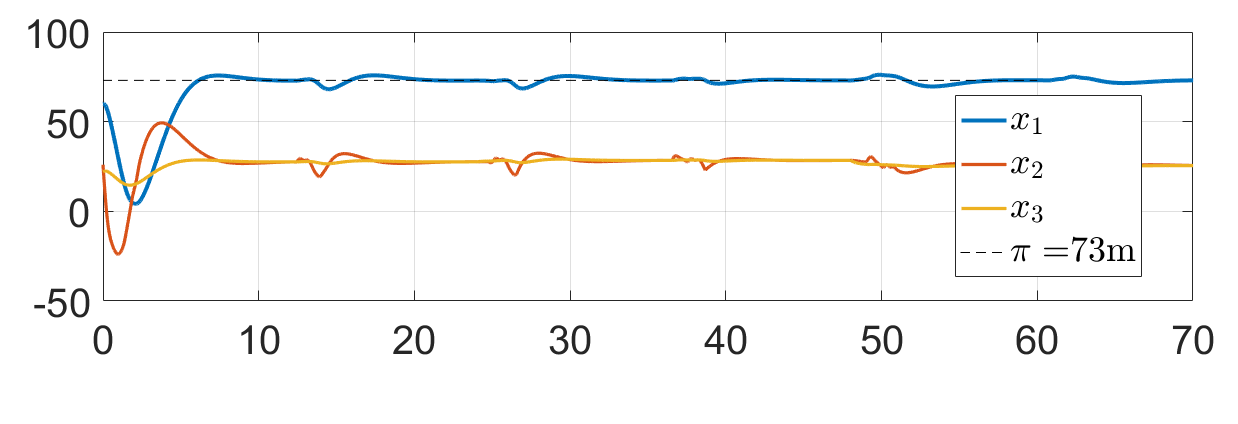}
\\
\vspace{-0.3cm}\includegraphics[width=1\columnwidth]{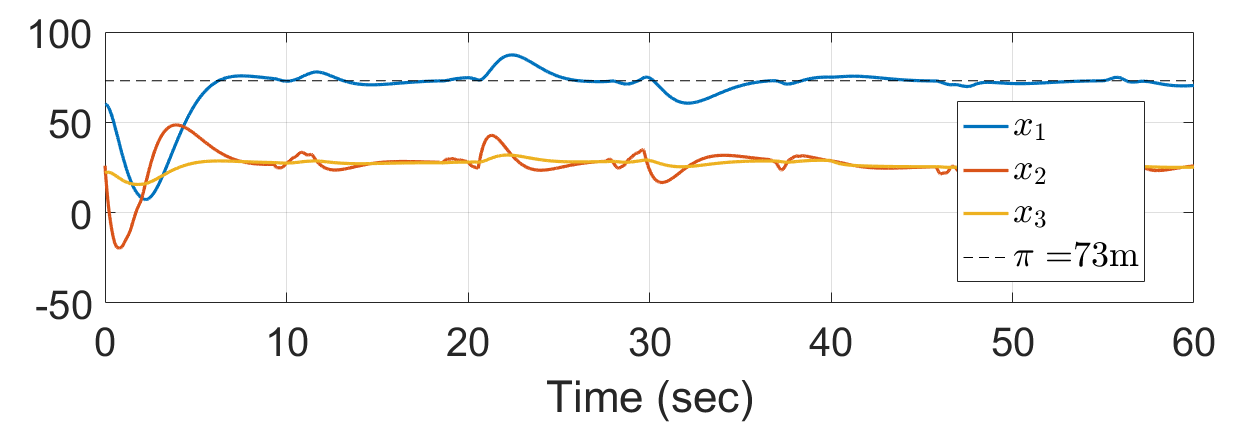}
\\
\vspace{-0.2cm}
	\caption{Convergence performance of system state. Top: results obtained using fully clean data.
Bottom: results obtained using partially mixed data.} 
\vspace{-0.3cm}
	\label{fig17}
\end{figure}


\vspace{-0.2cm}

\section{Concluding Remarks}
In this paper, we presented a periodic off-policy learning scheme to mitigate insider threats with varying intentions. We developed a proportional-integral-augmented delayed incremental design to eliminate steady-state bias and reject mode-dependent insider disturbances. By properly designing the inter-learning interval, the incremental learning process is ensured to operate effectively. Convergence guarantees for both the learning process and the closed-loop system are established, and the corresponding mitigation performance is characterized. A potential direction for future research should consider extending the proposed framework to multi-agent systems and studying the distributed mitigation learning.

\bibliographystyle{IEEEtran}
\bibliography{reference}

\end{document}